\documentclass[11pt]{article}%
\usepackage{amssymb}
\usepackage{amsmath}
\usepackage{amsfonts}
\usepackage{graphicx}%
\newtheorem{theorem}{Theorem}

\newtheorem{corollary}[theorem]{Corollary}

\newtheorem{definition}[theorem]{Definition}
\newtheorem{example}[theorem]{Example}

\newtheorem{lemma}[theorem]{Lemma}

\newtheorem{remark}[theorem]{Remark}

\newenvironment{proof}[1][Proof]{\noindent\textbf{#1.} }{\ \rule{0.5em}{0.5em}}
\begin{document}

\title{The Geometric Invariants of Null Cartan Curves Under The Similarity Transformations}
\author{Hakan Simsek
\and Mustafa \"{O}zdemir}
\maketitle

\begin{abstract}
In this paper, we study the differential geometry of null Cartan curves under
the similarity transformations in the Minkowski space-time. Besides, we extend
the fundamental theorem for a null Cartan curve according to a similarity
motion. We find the equations of all self-similar null curves which is given
its shape Cartan curvatures.

\qquad\ 



\end{abstract}

\section{\textbf{Introduction}}

\qquad A similarity transformation (or similitude) of Euclidean space, which
consists of a rotation, a translation and an isotropic scaling, is an
automorphism preserving the angles and ratios between lengths. The structure,
which forms geometric properties unchanged by similarity transformations, is
called the \textit{similarity geometry}. The whole Euclidean geometry can be
considered as a glass of similarity geometry. The similarity transformations
are studying in most area of the pure and applied mathematics.

\qquad Curve matching is an important research area in the computer vision and
pattern recognition, which can help us determine what category the given curve
belongs to. S. Li \cite{vision2, vision3} presented a system for matching and
pose estimation of 3D space curves under the similarity transformation. Brook
et al. \cite{vision00} discussed various problems of image processing and
analysis by using the similarity transformation. Sahbi \cite{vision0}
investigated a method for shape description based on kernel principal
component analysis (KPCA) in the similarity invariance of KPCA. On the other
hand, Chou and Qu \cite{chaos} showed that the motions of curves in two, three
and n-dimensional ($n>3$) similarity geometries correspond to the Burgers
hierarchy, Burgers-mKdV hierarchy and a multi-component generalization of
these hierarchies by using the similarity invariants of curves.

\qquad The idea of self-similarity is one of the most basic and fruitful ideas
in mathematics. A self-similar object is exactly similar to a part of itself,
which in turn remains similar to a smaller part of itself, and so on. In the
last few decades, the self-similarity notion led to the areas such as fractal
geometry, dynamical systems, computer networks and statistical physics.
Mandelbrot presented the first description of self-similar sets, namely sets
that may be expressed as unions of rescaled copies of themselves. He called
these sets fractals, which are systems that present such self-similar behavior
and the examples in nature are many. The Cantor set, the von Koch snowflake
curve and the Sierpinski gasket are some of the most famous examples of such
sets (see \cite{fractal1, fractal2, fractal3, fractal4}).

\qquad When the Euclidean space is endowed with the Lorentzian inner product,
we obtain the \textit{Lorentzian similarity geometry. }The Lorentzian flat
geometry\textit{ }is inside the\textit{ }Lorentzian similarity
geometr\textit{y. }Aristide \cite{simlorentz0} investigated the closed\textit{
}Lorentzian similarity manifolds.\textit{\ }Kamishima \cite{simlorentz}%
\textit{ }studied the properties of compact Lorentzian similarity manifolds
using developing maps and holonomy representations. The geometric invariants
of null curves in the Lorentzian similarity geometry have not been considered
so far.

\qquad Bonnor \cite{bonnor} introduced the Cartan frame to study the behaviors
of a null curve and proved the fundamental existence and congruence theorems
in Minkowski space-time. Bejancu \cite{null0} represented a method for the
general study of the geometry of null curves in Lorentz manifolds and, more
generally, in semi-Riemannian manifolds (see also the book \cite{duggal}).
Ferrandez, Gimenez and Lucas \cite{null1} gave a reference along a null curve
in an n-dimensional Lorentzian space. They showed the fundamental existence
and uniqueness theorems and described the null helices in higher dimensions.
C\"{o}ken and Ciftci \cite{coken} studied null curves in the Minkowski
space-time and characterized pseudo-spherical null curves and Bertrand null curves.

\qquad The study of the geometry of null curves has a growing importance in
the mathematical physics. The null curves use at the solution of some
equations in the classical relativistic string theory (see \cite{nullp0,
nullp1, nullp2}) Moreover, there exists a geometric particle model associated
with the geometry of null curves in the Minkowski space-time (see
\cite{nullpp, nullp}).

\qquad Berger \cite{berger} represented the broad content of similarity
transformations in the arbitrary dimensional Euclidean spaces. Encheva
and\textsc{ }Georgiev \cite{shape, similar} studied the differential geometric
invariants of curves according to a similarity in the finite dimensional
Euclidean spaces. The main idea of this paper is to introduce the differential
geometry of a null curve under the pseudo-similarity mapping and determine the
self-similar null curves in the Lorentzian similarity geometry.

\qquad The scope of paper is as follows. First, we give basic informations
about null Cartan curves. Then, we introduce a new parameter, which is called
pseudo-de Sitter parameter and is invariant under the similarity
transformation, to study null curves in Lorentzian similarity geometry. We
represent the differential geometric invariants of a null Cartan curve, which
are called shape Cartan curvatures, according to the group of similarity
transformations in the Minkowski space-time. We prove the uniqueness theorem
which states that two null Cartan curves having same the shape Cartan
curvatures are equivalent according to a similarity. Furthermore, we show the
existence theorem that is a process for constructing a null Cartan curve by
the shape Cartan curvatures under some initial conditions. Lastly, we obtain
equations of all self-similar null Cartan curves, whose shape Cartan
curvatures are real constant.

\section{Preliminaries}

\qquad Let $\mathbf{u}=\left(  u_{1},u_{2},u_{3},u_{4}\right)  $ and
$\mathbf{v}=\left(  v_{1},v_{2},v_{3},v_{4}\right)  $ be two arbitrary vectors
in Minkowski space-time $\mathbb{M}^{4}$. The Lorentzian inner product of
$\mathbf{u}$ and $\mathbf{v}$ can be stated as $\mathbf{u}\cdot\mathbf{v}%
=\mathbf{u}I^{\ast}\mathbf{v}^{T}$ where $I^{\ast}=diag(-1,1,1,1).$ We say
that a vector $\mathbf{u}$ in $\mathbb{M}^{4}$ is called spacelike, null
(lightlike) or timelike if $\mathbf{u\cdot u}>0,$ $\mathbf{u\cdot u}=0$ or
$\mathbf{u\cdot u}<0,$ respectively. The norm of the vector $\mathbf{u}$ is
represented by $\left\Vert \mathbf{u}\right\Vert =\sqrt{\left\vert
\mathbf{u}\cdot\mathbf{u}\right\vert }$.

\qquad We can describe the pseudo-hyperspheres in $\mathbb{M}^{4}$ as follows:
The \textit{hyperbolic 3-space} is defined by%
\[
H^{3}\left(  -1\right)  =\left\{  \mathbf{u}\in\mathbb{M}^{4}:\mathbf{u\cdot
u}=-1\right\}
\]
and \textit{de Sitter 3-space} is defined by
\[
S_{1}^{3}=\left\{  \mathbf{u}\in\mathbb{M}^{4}:\mathbf{u\cdot u}=1\right\}
\text{ \ (\cite{semi riemann, grub}).}%
\]

\qquad A basis $\boldsymbol{B}=\left\{  \mathbf{L},\mathbf{N},\mathbf{W}%
_{1},\mathbf{W}_{2}\right\}  $ of $\mathbb{M}^{4}$ is said to
\textit{pseudo-orthonormal} if it satisfies the following equations:%
\[
\mathbf{L\cdot L}=\mathbf{N\cdot N}=0,\text{ \ \ \ \ }\mathbf{L}%
\cdot\mathbf{N}=1,
\]%
\[
\mathbf{L\cdot W}_{i}=\mathbf{N\cdot W}_{i}=\mathbf{W}_{1}\cdot\mathbf{W}%
_{2}=0,
\]%
\[
\mathbf{W}_{i}\cdot\mathbf{W}_{i}=1
\]
where $i\in\left\{  1,2\right\}  $ (\cite{duggal})$.$

\qquad Now, we consider the mapping $\mathbf{\varphi}:\left(  \mathbf{\bar{L}%
},\mathbf{\bar{N}},\mathbf{\bar{W}}_{1},\mathbf{\bar{W}}_{2}\right)
\rightarrow\left(  \mathbf{L},\mathbf{N},\mathbf{W}_{1},\mathbf{W}_{2}\right)
$ of one pseudo-orthonormal basis onto another at any point $P$ in
$\mathbb{M}^{4},$ defined by
\begin{equation}%
\begin{bmatrix}
\mathbf{L}\\
\mathbf{N}\\
\mathbf{W}_{1}\\
\mathbf{W}_{2}%
\end{bmatrix}
=%
\begin{bmatrix}
\lambda & 0 & 0 & 0\\
-\frac{1}{2}\lambda\left(  \varepsilon^{2}+\zeta^{2}\right)  & \lambda^{-1} &
-\varepsilon & \zeta\\
\lambda\varepsilon\cos\theta+\lambda\zeta\sin\theta & 0 & \cos\theta &
-\sin\theta\\
\lambda\varepsilon\sin\theta-\lambda\zeta\cos\theta & 0 & \sin\theta &
\cos\theta
\end{bmatrix}%
\begin{bmatrix}
\mathbf{\bar{L}}\\
\mathbf{\bar{N}}\\
\mathbf{\bar{W}}_{1}\\
\mathbf{\bar{W}}_{2}%
\end{bmatrix}
\label{rot}%
\end{equation}
where $\gamma,\varepsilon,\zeta$ and $\theta$ are real constants and
$\lambda\neq0.$ The image of pseudo-orthonormal basis under the mapping
$\mathbf{\varphi}$ is a pseudo-orthonormal basis. Moreover, the orientation is
preserved by $\left(  \ref{rot}\right)  .$ Bonnor \cite{bonnor} defined the
mapping $\mathbf{\varphi}$ as a \textit{null rotation}. A null rotation at $P$
is equivalent to a Lorentzian transformation between two sets of natural
coordinate functions whose values coincide at $P$.

\qquad A curve locally parameterized by $\gamma:J\subset%
\mathbb{R}
\rightarrow\mathbb{M}^{4}$ is called a null curve if $\gamma^{\prime}(t)\neq0$
is a null vector for all $t$. We know that a\ null curve $\gamma(t)$ satisfies
$\gamma^{\prime\prime}(t)\cdot\gamma^{\prime\prime}(t)\geq0$ (\cite{duggal}).
If $\gamma^{\prime\prime}(t)\cdot\gamma^{\prime\prime}(t)=1,$ then it is said
that a null curve $\gamma(t)$ in $\mathbb{M}^{4}$ is parameterized by
\textit{pseudo-arc}. If we assume that the acceleration vector of the null
curve is not null, the pseudo-arc parametrization becomes as the following%
\begin{equation}
s=\int_{t_{0}}^{t}\left(  \gamma^{\prime\prime}(u)\cdot\gamma^{\prime\prime
}(u)\right)  ^{1/4}du\text{ \ (\cite{coken, bonnor}).} \label{n0}%
\end{equation}

\qquad A null curve $\gamma(t)$ in $\mathbb{M}^{4}$ with $\gamma^{\prime
\prime}(t)\cdot\gamma^{\prime\prime}(t)\neq0$ is a \textit{Cartan curve }if
$F_{\gamma}:=\left\{  \gamma^{\prime}(t),\gamma^{\prime\prime}(t),\gamma
^{(3)}(t),\gamma^{(4)}(t)\right\}  $ is linearly independent for any $t$.
There exists a unique \textit{Cartan frame} $C_{\gamma}:=\left\{
\mathbf{L},\mathbf{N},\mathbf{W}_{1},\mathbf{W}_{2}\right\}  $ of the Cartan
curve has the same orientation with $F_{\gamma}$ according to pseudo
arc-parameter $t,$ such that the following equations are satisfied;
\begin{align}
\gamma^{\prime} &  =\mathbf{L},\nonumber\\
\mathbf{L}^{\prime} &  =\mathbf{W}_{1},\nonumber\\
\mathbf{N}^{\prime} &  =\kappa\mathbf{W}_{1}+\tau\mathbf{W}_{2}\label{n1}\\
\mathbf{W}_{1}^{\prime} &  =-\kappa\mathbf{L}-\mathbf{N}\nonumber\\
\mathbf{W}_{2}^{\prime} &  =-\tau\mathbf{L}\nonumber
\end{align}
where $\mathbf{N}$ is a null vector, which is called null transversal vector
field, and $C_{\gamma}$ is pseudo-orthonormal and positively oriented. The
functions $\kappa$ and $\tau$ are called the \textit{Cartan curvatures }of
$\gamma\left(  t\right)  $ and their values are given as
\begin{align}
\kappa\left(  t\right)   &  =\frac{1}{2}\left(  \gamma^{(3)}(t)\cdot
\gamma^{(3)}(t)\right)  \label{n2}\\
\tau\left(  t\right)   &  =-\sqrt{\gamma^{(4)}(t)\cdot\gamma^{(4)}(t)-\left(
\gamma^{(3)}(t)\cdot\gamma^{(3)}(t)\right)  ^{2}}\nonumber
\end{align}
for the pseudo-arc parameter $t$. It can be seen the materials \cite{null0},
\cite{null1}, \cite{duggal} and \cite{bonnor} for more information about the
geometry of null curves.

\section{Geometric Invariants of Null Curves in Lorentzian Similarity
Geometry}

\qquad Now, we define a pseudo-similarity transformation for null curves in
$\mathbb{M}^{4}$. A\emph{ pseudo-similarity (p-similarity)} of Minkowski
space-time is a composition of a Lorentzian transformation (or null rotation),
translation and a scaling. Any p-similarity map $f:\mathbb{M}^{4}%
\rightarrow\mathbb{M}^{4}$ is determined by%

\begin{equation}
f\left(  x\right)  =\mu\mathbf{\varphi}\left(  x\right)  +\mathbf{b}%
,\label{05}%
\end{equation}
where $\mu\neq0$ is a real constant, $\mathbf{\varphi}$ is a null rotation and
$\mathbf{b}$ is a translation vector. The p-similarity transformations are a
group under the composition of maps and denoted by \textbf{Simp}$\left(
\mathbb{M}^{4}\right)  $. This group is a fundamental group of the Lorentzian
similarity geometry spanned by the pseudo-orthonormal basis. The p-similarity
transformations in $\mathbb{M}^{4}$ preserve the orientation.

\qquad Let $\gamma\left(  t\right)  :J\subset%
\mathbb{R}
\rightarrow\mathbb{M}^{4}$ be a null curve in $\mathbb{M}^{4}$. We denote
image of $\gamma$ under $f\in$ \textbf{Simp}$\left(  \mathbb{M}^{4}\right)  $
by $\beta$. Then, the null curve $\beta$ can be stated as
\begin{equation}
\beta\left(  t\right)  =\mu\mathbf{\varphi}\gamma\left(  t\right)  +b,\text{
\ \ \ \ \ \ }t\in J. \label{1}%
\end{equation}
The pseudo-arc length function $\beta$ starting at $t_{0}\in J$ is%
\begin{equation}
s^{\ast}\left(  t\right)  =\int\limits_{t_{0}}^{t}\left(  \beta^{\prime\prime
}(u)\cdot\beta^{\prime\prime}(u)\right)  ^{1/4}du=\sqrt{\mu}s\left(  t\right)
\label{1'}%
\end{equation}
where $s\in I\subset%
\mathbb{R}
$ is pseudo-arc parameter of $\gamma:I\rightarrow\mathbb{M}^{4}.$ From now on,
we will denote by a prime \textquotedblright$^{\prime}$\textquotedblright\ the
differentiation with respect to $s$. We can compute the Cartan curvatures
$\kappa_{\beta}\left(  \sqrt{\mu}s\right)  $ and $\tau_{\beta}\left(
\sqrt{\mu}s\right)  $ of $\beta$ by using $\left(  \ref{n2}\right)  $ as
\begin{equation}
\kappa_{\beta}=\frac{1}{\mu}\kappa_{\gamma},\text{ \ \ \ \ }\tau_{\beta}%
=\frac{1}{\mu}\tau_{\gamma}. \label{2}%
\end{equation}

\qquad We define $\mathbf{W}_{2}-$indicatrix $\gamma_{W_{2}}$ of the null
curve $\gamma$ parameterized by $\gamma_{W_{2}}\left(  s\right)
=\mathbf{W}_{2}\left(  s\right)  $. The $\mathbf{W}_{2}-$indicatrix is a
pseudo-hyperspherical curve lies on the de Sitter 3-space $S_{1}^{3}$. Since
the curve $\gamma_{W_{2}}$ is a null curve, the pseudo-arc parameter
$\sigma_{\gamma}$ of $\gamma_{W_{2}}$ can be given as $d\sigma_{\gamma}%
=\sqrt{\tau_{\gamma}}ds$ by using the equation $\left(  \ref{n0}\right)  .$
The parameter $\sigma_{\gamma}$ is invariant under the p-similarity
transformation since it can be easily found $d\sigma_{\beta}=d\sigma_{\gamma
},$ where $\sigma_{\beta}$ is the pseudo-de Sitter parameter of $\beta$.
Therefore, we can reparametrize a null curve with the pseudo-de Sitter
parameter in order to study differential geometry of a null curve under the
p-similarity transformation. The parameter $\sigma_{\gamma}$ is called
\textit{pseudo-de Sitter parameter} of $\gamma.$

\qquad The derivative formulas of $\gamma$ and $C_{\gamma}$ with respect to
$\sigma_{\gamma}$ are given by
\begin{equation}
\frac{d\gamma}{d\sigma_{\gamma}}=\frac{1}{\sqrt{\tau_{\gamma}}}\mathbf{L}%
,\text{ \ \ \ \ \ }\frac{d^{2}\gamma}{d\sigma_{\gamma}^{2}}=\frac
{-d\tau_{\gamma}}{2\tau_{\gamma}d\sigma_{\gamma}}\frac{d\gamma}{d\sigma
_{\gamma}}+\frac{1}{\tau_{\gamma}}\mathbf{W}_{1}\label{4}%
\end{equation}
and
\begin{align}
\frac{d\mathbf{L}}{d\sigma_{\gamma}} &  =\frac{1}{\sqrt{\tau_{\gamma}}%
}\mathbf{W}_{1}\nonumber\\
\frac{d\mathbf{N}}{d\sigma_{\gamma}} &  =\frac{\kappa_{\gamma}}{\sqrt
{\tau_{\gamma}}}\mathbf{W}_{1}+\sqrt{\tau_{\gamma}}\mathbf{W}_{2}\label{n3}\\
\frac{d\mathbf{W}_{1}}{d\sigma_{\gamma}} &  =-\frac{\kappa_{\gamma}}%
{\sqrt{\tau_{\gamma}}}\mathbf{L}-\frac{1}{\sqrt{\tau_{\gamma}}}\mathbf{N}%
\nonumber\\
\frac{d\mathbf{W}_{2}}{d\sigma_{\gamma}} &  =-\sqrt{\tau_{\gamma}}%
\mathbf{L.}\nonumber
\end{align}
Similarly, we can find the same formulas $\left(  \ref{4}\right)  $ and
$\left(  \ref{n3}\right)  $ for the null curve $\beta$.

\qquad Now, we construct a new frame corresponding to p-similarity
transformation for a null curve. Let's denote the functions
\[
\tilde{\tau}_{\gamma}:=\frac{-d\tau_{\gamma}}{2\tau_{\gamma}d\sigma_{\gamma}%
}\text{ and }\tilde{\kappa}_{\gamma}:=\frac{\kappa_{\gamma}}{\tau_{\gamma}},
\]
respectively. The functions $\tilde{\tau}_{\gamma}$ and $\tilde{\kappa
}_{\gamma}$ are invariant under the p-similarity because of $\tilde{\tau
}_{\beta}=\tilde{\tau}_{\gamma}$ and $\tilde{\kappa}_{\beta}=\tilde{\kappa
}_{\gamma}$. Let be
\begin{align*}
\mathbf{L}^{sim}  &  =\sqrt{\tau_{\gamma}}\mathbf{L},\text{ \ \ \ \ \ \ }%
\mathbf{N}^{sim}=\frac{1}{\sqrt{\tau_{\gamma}}}\mathbf{N}\\
\mathbf{W}_{1}^{sim}  &  =\mathbf{W}_{1},\text{ \ \ \ \ \ }\mathbf{W}%
_{2}^{sim}=\mathbf{W}_{2}%
\end{align*}
such that the equality $\mathbf{L}^{sim}\cdot\mathbf{N}^{sim}=1$ is satisfied
and $C_{\gamma}^{sim}:=\left\{  \mathbf{L}^{sim},\mathbf{N}^{sim}%
,\mathbf{W}_{1}^{sim},\mathbf{W}_{2}^{sim}\right\}  $ is a pseudo-orthonormal
frame of $\gamma$. Then, the derivative formulas for $C_{\gamma}^{sim}$ are
\begin{equation}
\frac{d}{d\sigma_{\gamma}}\left(  C_{\gamma}^{sim}\right)  ^{T}=P\left(
C_{\gamma}^{sim}\right)  ^{T} \label{n4}%
\end{equation}
where
\[
P=%
\begin{bmatrix}
0 & 0 & 1 & 0\\
0 & 0 & \tilde{\kappa}_{\gamma} & 1\\
-\tilde{\kappa}_{\gamma} & -1 & 0 & 0\\
-1 & 0 & 0 & 0
\end{bmatrix}
.
\]

\qquad We consider the pseudo-orthogonal frame $C_{\gamma}^{H}:=\left\{
\mathbf{H}_{1}^{\gamma},\mathbf{H}_{2}^{\gamma},\mathbf{H}_{3}^{\gamma
},\mathbf{H}_{4}^{\gamma}\right\}  $ for the null curve $\gamma$ where
\[
\mathbf{H}_{1}^{\gamma}=\frac{1}{\tau_{\gamma}}\mathbf{L}^{sim},\text{
}\mathbf{H}_{2}^{\gamma}=\frac{1}{\tau_{\gamma}}\mathbf{N}^{sim},\text{
}\mathbf{H}_{3}^{\gamma}=\frac{1}{\tau_{\gamma}}\mathbf{W}_{1}^{sim}\text{ and
}\mathbf{H}_{4}^{\gamma}=\frac{1}{\tau_{\gamma}}\mathbf{W}_{2}^{sim}.
\]
Since we can obtain $f(\mathbf{H}_{i}^{\gamma})=\mathbf{H}_{i}^{\beta},$
$i=1,\cdots,4,$ from $\left(  \ref{05}\right)  ,$ the pseudo-orthogonal frame
$C_{\gamma}^{H}$ is invariant according to p-similarity map. Then, using
$\left(  \ref{4}\right)  $ and $\left(  \ref{n4}\right)  ,$ we get the
derivative formulas of $C_{\gamma}^{H}$ as the following
\begin{equation}
\frac{d}{d\sigma}\left(  C_{\gamma}^{H}\right)  ^{T}=\tilde{P}\left(
C_{\gamma}^{H}\right)  ^{T} \label{n5}%
\end{equation}
where
\[
\tilde{P}=%
\begin{bmatrix}
2\tilde{\tau}_{\gamma} & 0 & 1 & 0\\
0 & 2\tilde{\tau}_{\gamma} & \tilde{\kappa}_{\gamma} & 1\\
-\tilde{\kappa}_{\gamma} & -1 & 2\tilde{\tau}_{\gamma} & 0\\
-1 & 0 & 0 & 2\tilde{\tau}_{\gamma}%
\end{bmatrix}
\]

\qquad We can think the equation $\left(  \ref{n5}\right)  $ as the structure
equation of a null curve $\gamma$ according to the pseudo-orthogonal moving
frame $C_{\gamma}^{H}$ and the p-similarity group \textbf{Simp}$\left(
\mathbb{M}^{4}\right)  $. As a result, the following lemma is obtained.

\begin{lemma}
Let $\gamma:I\rightarrow\mathbb{M}^{4}$ be a null Cartan curve with pseudo-de
Sitter parameter $\sigma$ and $\left\{  \kappa_{\gamma},\tau_{\gamma}\right\}
$ be Cartan curvatures of $\gamma$ with the Cartan frame $C_{\gamma}.$ Then,
the functions
\begin{equation}
\tilde{\tau}_{\gamma}=\frac{-d\tau_{\gamma}}{2\tau_{\gamma}d\sigma_{\gamma}%
},\text{ \ \ \ }\tilde{\kappa}_{\gamma}=\frac{\kappa_{\gamma}}{\tau_{\gamma}}
\label{n6}%
\end{equation}
and the pseudo-orthogonal frame $C_{\gamma}^{H}$ are invariant under the
p-similarity transformation\ in the Minkowski space-time and the derivative
formulas of $C_{\gamma}^{H}$ with respect to $\sigma$ are given by the
equation $\left(  \ref{n5}\right)  .$
\end{lemma}

\begin{definition}
The functions $\tilde{\tau}_{\gamma}=\frac{-d\tau_{\gamma}}{2\tau_{\gamma
}d\sigma_{\gamma}},$ $\tilde{\kappa}_{\gamma}=\frac{\kappa_{\gamma}}%
{\tau_{\gamma}}$ and the pseudo-orthogonal frame $C_{\gamma}^{H}$ are called
\textit{shape Cartan curvatures} and \textit{shape Cartan frame }of a null
Cartan curve $\gamma,$ respectively.
\end{definition}

\begin{remark}
We consider the $\mathbf{W}_{1}-$indicatrix $\gamma_{W_{1}}$ of null curve
$\gamma$ parameterized by $\gamma_{W_{1}}\left(  s\right)  =\mathbf{W}%
_{1}\left(  s\right)  ,$ where $s$ is a pseudo-arc parameter of $\gamma.$ The
curve $\gamma_{W_{1}}$ is a pseudo-hyperspherical spacelike curve if
$\kappa_{\gamma}>0$ or pseudo-hyperspherical timelike curve if $\kappa
_{\gamma}<0$ on $S_{1}^{3}.$ If $u$ is a arc-parameter of $\gamma_{W_{1}},$
then we can find $du=\sqrt{\left\vert 2\kappa_{\gamma}\right\vert }ds.$ The
parameter $u$ is invariant according to p-similarity transformation;
therefore, it can also be used this parametrization for a null Cartan curve in
Lorentzian similarity geometry.
\end{remark}

\begin{remark}
We take $\mathbf{W}_{1}^{sim}=\dfrac{d\mathbf{L}^{sim}}{d\sigma_{\gamma}}$
instead of $\mathbf{W}_{1}^{sim}=\mathbf{W}_{1}$ in the frame $C_{\gamma
}^{sim}.$ The derivative formulas for a new frame are
\[
\frac{d}{d\sigma_{\gamma}}\left(  C_{\gamma}^{sim}\right)  ^{T}=%
\begin{bmatrix}
0 & 0 & 1 & 0\\
0 & 0 & \tilde{\xi}_{\gamma} & 1\\
-\tilde{\xi}_{\gamma} & -1 & 0 & 0\\
-1 & 0 & 0 & 0
\end{bmatrix}
\left(  C_{\gamma}^{sim}\right)  ^{T}%
\]
where $\tilde{\xi}_{\gamma}=-\tilde{\tau}_{\gamma}^{2}+\tilde{\kappa}_{\gamma
}^{2}.$ It may be considered an alternative frame for a null Cartan curve in
the Lorentzian similarity geometry. However, the problem is that although
$\mathbf{W}_{1}^{sim}$ is a unit spacelike vector, the new frame is not the
pseudo-orthogonal due to $\left\langle \mathbf{W}_{1}^{sim},\mathbf{N}%
^{sim}\right\rangle \neq0.$
\end{remark}

\section{The Fundamental Theorem for a Null Curve in Lorentzian Similarity
Geometry}

\qquad The existence and uniqueness theorems are shown by \cite{null0, null1}
and \cite{bonnor} for a null Cartan curve under the Lorentz transformations.
This notion can be extended with respect to \textbf{Simp}$\left(
\mathbb{M}^{4}\right)  $ for the null Cartan curves parameterized by the
pseudo-de Sitter parameter.

\begin{theorem}
\label{teklik}(Uniqueness Theorem) Let $\gamma,\beta:I\rightarrow
\mathbb{M}^{4}$ be two null Cartan curves parameterized by the same pseudo-de
Sitter parameter $\sigma$, where $I\subset%
\mathbb{R}
$ is an open interval. Suppose that $\gamma$ and $\beta$ have the same shape
Cartan curvatures $\tilde{\tau}_{\gamma}=\tilde{\tau}_{\beta}$ and
$\tilde{\kappa}_{\gamma}=\tilde{\kappa}_{\beta}$ for any $\sigma\in I.$ Then,
there exists a $f\in$\textbf{Simp}$\left(  \mathbb{M}^{4}\right)  $ such that
$\beta=f\circ\gamma.$
\end{theorem}

\begin{proof}
Let $\kappa_{\gamma},$ $\tau_{\gamma}$ and $\kappa_{\beta},$ $\tau_{\beta}$ be
the Cartan curvatures and also $s$ and $s^{\ast}$ be the pseudo-arc length
parameters of $\gamma$ and $\beta,$ respectively. Using the equality
$\tilde{\tau}_{\gamma}=\tilde{\tau}_{\beta},$ we get $\tau_{\gamma}=\mu
\tau_{\beta}$ for some real constant $\mu>0.$ Then, the equality
$\tilde{\kappa}_{\gamma}=\tilde{\kappa}_{\beta}$ implies $\kappa_{\gamma}%
=\mu\kappa_{\beta}.$ On the other hand, we can write $ds=\frac{1}{\sqrt{\mu}%
}ds^{\ast}$ from the definition of pseudo-de Sitter parameter $\sigma.$

\qquad Let's consider the map $\Psi:\mathbb{M}^{4}\rightarrow\mathbb{M}^{4}$
defined by $\Psi\left(  x\right)  =\frac{1}{\mu}\mathbf{\varphi}\left(
x\right)  $ where $\mathbf{\varphi}$ is a null rotation. Using the equation
$\left(  \ref{2}\right)  ,$ the null Cartan curves $\alpha=\Psi\left(
\beta\right)  $ and $\gamma$ have the same Cartan curvatures. Then, there
exists a Lorentzian transformation $\phi:\mathbb{M}^{4}\rightarrow
\mathbb{M}^{4}$ according to the uniqueness theorem for the null Cartan curves
(see \cite{null0, null1}) such that $\phi\left(  \gamma\right)  =\alpha.$
Therefore, we have a transformation $f=\Psi^{-1}\circ\phi:\mathbb{M}%
^{4}\rightarrow\mathbb{M}^{4}$ which is a p-similarity and $f\left(
\gamma\right)  =\beta.$
\end{proof}

\qquad The following theorem shows that every two functions determine a null
Cartan curve according to a p-similarity under some initial conditions.

\begin{theorem}
\label{varl}(Existence Theorem) Let $z_{i}:I\rightarrow%
\mathbb{R}
,$ $i=1,2$ be two functions and $\mathbf{L}^{0sim},$ $\mathbf{N}^{0sim},$
$\mathbf{W}_{1}^{0sim},$ $\mathbf{W}_{2}^{0sim}$ be a pseodo-orthonormal frame
at a point $x_{0}$ in the Minkowski space-time. According to a p-similarity
with the center $x_{0}$ there exists a unique null Cartan curve $\gamma
:I\rightarrow\mathbb{M}^{4}$ parameterized by a pseudo-de Sitter parameter
$\sigma$ such that $\gamma$ satisfies the following conditions:

$\left(  i\right)  $ There exists $\sigma_{0}\in I$ such that $\gamma\left(
\sigma_{0}\right)  =x_{0}$ and the shape Cartan frame of $\gamma$ at $x_{0}$
is $\mathbf{L}^{0sim},$ $\mathbf{N}^{0sim},$ $\mathbf{W}_{1}^{0sim},$
$\mathbf{W}_{2}^{0sim}.$

$\left(  ii\right)  $ $\tilde{\kappa}_{\gamma}\left(  \sigma\right)
=z_{1}\left(  \sigma\right)  $ and $\tilde{\tau}_{\gamma}\left(
\sigma\right)  =z_{2}\left(  \sigma\right)  ,$ for any $\sigma\in I.$
\end{theorem}

\begin{proof}
Let us consider the following system of differential equations with respect to
a matrix-valued function $\mathbf{K}\left(  \sigma\right)  =\left(
\mathbf{L}^{sim},\mathbf{N}^{sim},\mathbf{W}_{1}^{sim},\mathbf{W}_{2}%
^{sim}\right)  ^{T}$%
\begin{equation}
\frac{d\mathbf{K}}{d\sigma}\left(  \sigma\right)  =\mathbf{M}\left(
\sigma\right)  \mathbf{K}\left(  \sigma\right)  \label{f4}%
\end{equation}
with a given matrix
\[
\mathbf{M}\left(  \sigma\right)  =%
\begin{bmatrix}
0 & 0 & 1 & 0\\
0 & 0 & z_{1} & 1\\
-z_{1} & -1 & 0 & 0\\
-1 & 0 & 0 & 0
\end{bmatrix}
.
\]
The system $\left(  \ref{f4}\right)  $ has a unique solution $\mathbf{W}%
\left(  \sigma\right)  $ which satisfies the initial conditions
\[
\mathbf{K}\left(  \sigma_{0}\right)  =\left(  \mathbf{L}^{0sim},\mathbf{N}%
^{0sim},\mathbf{W}_{1}^{0sim},\mathbf{W}_{2}^{0sim}\right)  ^{T}%
\]
for $\sigma_{0}\in I.$ If $\mathbf{K}^{t}\left(  \sigma\right)  $ is the
transposed matrix of $\mathbf{K}\left(  \sigma\right)  ,$ then
\begin{align*}
\frac{d}{d\sigma}\left(  \mathbf{J}^{\ast}\mathbf{K}^{t}\mathbf{J}^{\ast
}\mathbf{K}\right)   &  =\mathbf{J}^{\ast}\frac{d}{d\sigma}\mathbf{K}%
^{t}\mathbf{J}^{\ast}\mathbf{K}+\mathbf{J}^{\ast}\mathbf{K}^{t}\mathbf{J}%
^{\ast}\frac{d}{d\sigma}\mathbf{K}\\
&  =\mathbf{J}^{\ast}\mathbf{K}^{t}\mathbf{M}^{t}\mathbf{J}^{\ast}%
\mathbf{K}+\mathbf{J}^{\ast}\mathbf{K}^{t}\mathbf{J}^{\ast}\mathbf{MK}\\
&  =\mathbf{J}^{\ast}\mathbf{K}^{t}\left(  \mathbf{M}^{t}\mathbf{J}^{\ast
}+\mathbf{J}^{\ast}\mathbf{M}\right)  \mathbf{K}=0
\end{align*}
since we have the equation $\mathbf{M}^{t}\mathbf{J}^{\ast}+\mathbf{J}^{\ast
}\mathbf{M}=%
\begin{bmatrix}
0
\end{bmatrix}
_{4\times4}$ where
\[
\mathbf{J}^{\ast}=%
\begin{bmatrix}
0 & 1 & 0 & 0\\
1 & 0 & 0 & 0\\
0 & 0 & 1 & 0\\
0 & 0 & 0 & 1
\end{bmatrix}
.
\]
Also, we have $\mathbf{J}^{\ast}\mathbf{W}^{t}\left(  \sigma_{0}\right)
\mathbf{J}^{\ast}\mathbf{W}\left(  \sigma_{0}\right)  =\mathbf{I}$ where
$\mathbf{I}$ is the unit matrix since $\mathbf{L}^{0sim},$ $\mathbf{N}%
^{0sim},$ $\mathbf{W}_{1}^{0sim},$ $\mathbf{W}_{2}^{0sim}$ is the
pseudo-orthonormal 4-frame. As a result, we find $\mathbf{J}^{\ast}%
\mathbf{X}^{t}\left(  \sigma\right)  \mathbf{J}^{\ast}\mathbf{X}\left(
\sigma\right)  =\mathbf{I}$ for any $\sigma\in I.$ It means that the vector
fields $\mathbf{L}^{sim},$ $\mathbf{N}^{sim},$ $\mathbf{W}_{1}^{sim},$ and
$\mathbf{W}_{2}^{sim}$ form pseudo-orthonormal frame field in the Minkowski space-time.

\qquad Let $\gamma:I\rightarrow\mathbb{M}^{4}$ be the null curve given by
\begin{equation}
\gamma\left(  \sigma\right)  =x_{0}+\int_{\sigma_{0}}^{\sigma}e^{2\int
z_{2}\left(  \sigma\right)  d\sigma}\mathbf{L}^{sim}\left(  \sigma\right)
d\sigma,\text{ \ \ \ \ \ \ \ \ }\sigma\in I. \label{d}%
\end{equation}
Using the equality $\left(  \ref{f4}\right)  $, we get that $\gamma\left(
\sigma\right)  $ is a null Cartan curve in Minkowski space-time with shape
Cartan curvatures $\tilde{\kappa}_{\gamma}\left(  \sigma\right)  =z_{1}\left(
\sigma\right)  $ and $\tilde{\tau}_{\gamma}\left(  \sigma\right)
=z_{2}\left(  \sigma\right)  .$ Also, we find $d\sigma=e^{-\int z_{2}\left(
\sigma\right)  d\sigma}ds$ by using $\left(  \ref{n0}\right)  $ and $\left(
\ref{f4}\right)  ,$ where $s$ is a pseudo-arc parameter; thus, $\sigma$ is the
pseudo-de Sitter parameter of the null Cartan curve $\gamma$. Besides, the
pseudo-orthonormal 4-frame $\left\{  \mathbf{L}^{sim},\mathbf{N}%
^{sim},\mathbf{W}_{1}^{sim},\mathbf{W}_{2}^{sim}\right\}  $ is a Cartan frame
of the null Cartan curve $\gamma$ under the p-similarity transformation.
\end{proof}

\begin{corollary}
In case of $\tilde{\tau}_{\gamma}\left(  \sigma\right)  =0,$ the Cartan
curvature $\tau_{\gamma}=c$ is a non-zero real constant. Then, the
parametrization of a null curve $\gamma:I\rightarrow\mathbb{M}^{4}$ with
$\tilde{\tau}_{\gamma}\left(  \sigma\right)  =0$ with respect to pseudo-de
Sitter parameter $\sigma$ is given by
\begin{equation}
\gamma\left(  \sigma\right)  =x_{0}+\frac{1}{c}\int_{\sigma_{0}}^{\sigma
}\mathbf{L}^{sim}\left(  \sigma\right)  d\sigma,\text{ \ \ \ \ \ \ \ \ }%
\sigma\in I \label{d1}%
\end{equation}
from the equation $\left(  \ref{4}\right)  $ and $\left(  \ref{d}\right)  .$
\end{corollary}

\begin{example}
\label{ex}Let shape Cartan curvatures of a null curve $\gamma:I\rightarrow
\mathbb{M}^{4}$ be $\tilde{\kappa}_{\gamma}=0$ and $\tilde{\tau}_{\gamma
}=\frac{1}{\sigma}$. Choose the initial conditions
\begin{align}
\mathbf{L}^{0sim}  &  =\left(  \frac{1}{\sqrt{2}},0,\frac{1}{\sqrt{2}%
},0\right)  \mathbf{,\mathbf{N}}^{0sim}=\left(  \frac{-1}{\sqrt{2}},0,\frac
{1}{\sqrt{2}},0\right)  ,\label{n8}\\
\mathbf{W}_{1}^{0sim}  &  =\left(  0,\frac{1}{\sqrt{2}},0,\frac{1}{\sqrt{2}%
}\right)  ,\mathbf{W}_{2}^{0sim}=\left(  0,\frac{-1}{\sqrt{2}},0,\frac
{1}{\sqrt{2}}\right)  .\nonumber
\end{align}
Then, the system $\left(  \ref{f4}\right)  $ determine a null vector
$\mathbf{L}^{sim}$ given by%
\begin{equation}
\mathbf{L}^{sim}\left(  \sigma\right)  =\frac{1}{\sqrt{2}}\left(  \cosh
\sigma,\sinh\sigma,\cos\sigma,\sin\sigma\right)  \label{n9}%
\end{equation}
with $\mathbf{L}^{sim}\left(  0\right)  =\mathbf{L}^{0sim},$ in $\mathbb{M}%
^{4}$. Solving the equation $\left(  \ref{d}\right)  $ we obtain the null
Cartan curve $\gamma$ parameterized by
\[
\gamma\left(  \sigma\right)  =\frac{1}{\sqrt{2}}(\left(  \sigma^{2}+2\right)
\sinh\sigma-2\sigma\cosh\sigma,\left(  \sigma^{2}+2\right)  \cosh
\sigma-2\sigma\sinh\sigma,
\]%
\[
\left(  \sigma^{2}-2\right)  \sin\sigma+2\sigma\cos\sigma,\left(  2-\sigma
^{2}\right)  \cos\sigma+2\sigma\sin\sigma)
\]
for any $\sigma\in I.$
\end{example}

\section{Self-similar Null Cartan Curves}

\qquad A null Cartan curve $\gamma:I\rightarrow\mathbb{M}^{4}$ is called
\emph{self-similar }if any p-similarity $f\in G$ conserve globally $\gamma$
and $G$ acts transitively on $\gamma$ where $G$ is a one-parameter subgroup of
\textbf{Simp}$\left(  \mathbb{M}^{4}\right)  .$ This means that shape Cartan
curvatures $\tilde{\kappa}_{\gamma}$ and $\tilde{\tau}_{\gamma}$ are constant.
In fact, let $p=\gamma\left(  s_{1}\right)  $ and $q=\gamma\left(
s_{2}\right)  $ be two different points lying on $\gamma$ for any $s_{1}%
,s_{2}\in I.$ Since $G$ acts transitively on $\gamma,$ there is a p-similarity
$f\in G$ such that $f\left(  p\right)  =q$. Then, we find $\tilde{\kappa
}_{\gamma}\left(  s_{1}\right)  =\tilde{\kappa}_{\gamma}\left(  s_{2}\right)
$ and $\tilde{\tau}_{\gamma}\left(  s_{1}\right)  =\tilde{\tau}_{\gamma
}\left(  s_{2}\right)  ,$ which it implies the invariability of shape Cartan curvatures.

\qquad Now, we determine the parametrizations of all self-similar null curves
by means of the constant shape Cartan curvatures in the Minkowski space-time.
It can be separated to the four different cases as the following. We can take
the initial conditions $\left(  \ref{n8}\right)  $ in the example \ref{ex} for
the all cases.

\textbf{Case 1: }Let's take $\tilde{\kappa}_{\gamma_{1}}=0$ and $\tilde{\tau
}_{\gamma_{1}}=0.$ Then, using the equation $\left(  \ref{f4}\right)  $ we
find the null vector $\mathbf{L}^{sim}$ in the equation $\left(
\ref{n9}\right)  $. From the equation $\left(  \ref{d1}\right)  ,$ we obtain
the self-similar null curve parameterized by
\[
\gamma_{1}\left(  \sigma\right)  =\frac{1}{c\sqrt{2}}\left(  \sinh\sigma
,\cosh\sigma,\sin\sigma,-\cos\sigma\right)  .
\]

\textbf{Case 2: }Let's take $\tilde{\kappa}_{\gamma_{2}}=0$ and $\tilde{\tau
}_{\gamma_{2}}=b\neq0.$ Then, using the equation $\left(  \ref{f4}\right)  $
we find the null vector $\mathbf{L}^{sim}$%
\[
\mathbf{L}^{sim}\left(  \sigma\right)  =\frac{e^{2b\sigma}}{2\sqrt{2}}\left(
\cosh\sigma,\sinh\sigma,\cos\sigma,\sin\sigma\right)
\]
and from the equation $\left(  \ref{d}\right)  ,$ we get the self-similar null
curve given by
\begin{align*}
\gamma_{2}\left(  \sigma\right)   &  =\frac{1}{2\sqrt{2}}(\frac{\cosh\left(
w_{1}\sigma\right)  +\sinh\left(  w_{1}\sigma\right)  }{2b+1}+\frac
{\cosh\left(  w_{2}\sigma\right)  +\sinh\left(  w_{2}\sigma\right)  }{2b-1},\\
&  \frac{\cosh\left(  w_{1}\sigma\right)  +\sinh\left(  w_{1}\sigma\right)
}{2b+1}-\frac{\cosh\left(  w_{2}\sigma\right)  +\sinh\left(  w_{2}%
\sigma\right)  }{2b-1},
\end{align*}%
\[
\frac{4be^{2b\sigma}\cos\sigma+2e^{2b\sigma}\sin\sigma}{4b^{2}+1}%
,\frac{-2e^{2b\sigma}\cos\sigma+4be^{2b\sigma}\sin\sigma}{4b^{2}+1})
\]
where $w_{1}=\left(  2b+1\right)  $ and $w_{2}=\left(  2b-1\right)  .$

\textbf{Case 3: }Let's take $\tilde{\kappa}_{\gamma_{3}}=a\neq0$ and
$\tilde{\tau}_{\gamma_{3}}=0.$ Then, using the equation $\left(
\ref{f4}\right)  $ and $\left(  \ref{d1}\right)  $ it can be obtained the
self-similar null curve given by%
\[
\gamma_{3}\left(  \sigma\right)  =\frac{1}{c\sqrt{2}}\left(  \frac
{\sinh\left(  q_{1}\sigma\right)  }{q_{1}},\frac{\cosh\left(  q_{1}%
\sigma\right)  }{q_{1}},\frac{\sin\left(  q_{2}\sigma\right)  }{q_{2}}%
,\frac{-\cos\left(  q_{2}\sigma\right)  }{q_{2}}\right)
\]
where $q_{1}=\sqrt{-a+\sqrt{a^{2}+1}}$ and $q_{2}=\sqrt{a+\sqrt{a^{2}+1}}.$

\textbf{Case4: }Let's take $\tilde{\kappa}_{\gamma_{4}}=a\neq0$ and
$\tilde{\tau}_{\gamma_{4}}=b\neq0.$ Then, using the equation $\left(
\ref{f4}\right)  $ and $\left(  \ref{d}\right)  $ we obtain the self-similar
null curve given as%
\begin{align*}
\gamma_{4}\left(  \sigma\right)   &  =\frac{1}{2\sqrt{2}}(\frac{\cosh\left(
m_{1}\sigma\right)  +\sinh\left(  m_{1}\sigma\right)  }{m_{1}}+\frac
{\cosh\left(  m_{2}\sigma\right)  +\sinh\left(  m_{2}\sigma\right)  }{m_{2}%
},\\
&  \frac{\cosh\left(  m_{1}\sigma\right)  +\sinh\left(  m_{1}\sigma\right)
}{m_{1}}-\frac{\cosh\left(  m_{2}\sigma\right)  +\sinh\left(  m_{2}%
\sigma\right)  }{m_{2}},
\end{align*}%
\[
\frac{4be^{2b\sigma}\cos\left(  q_{2}\sigma\right)  +2q_{2}e^{2b\sigma}%
\sin\left(  q_{2}\sigma\right)  }{4b^{2}+q_{2}^{2}},\frac{-2q_{2}e^{2b\sigma
}\cos\left(  q_{2}\sigma\right)  +4be^{2b\sigma}\sin\left(  q_{2}%
\sigma\right)  }{4b^{2}+q_{2}^{2}})
\]
where $m_{1}=2b+q_{1},$ and $m_{2}=2b-q_{1}.$

\qquad A null curve is called a \textit{null helix }if it has the constant
Cartan curvatures not both zero in $\mathbb{M}^{4}$. The equations of null
helices satisfying $\tau\neq0$ are expressed by
\begin{equation}
\alpha\left(  s\right)  =\sqrt{\frac{1}{v^{2}+r^{2}}}\left(  \frac{1}{v}\sinh
vs,\frac{1}{v}\cosh vs,\frac{1}{r}\sin rs,-\frac{1}{r}\cos rs\right)
\label{n11}%
\end{equation}
where $v=\sqrt{\sqrt{\kappa^{2}+\tau^{2}}-\kappa}$ and $r=\sqrt{\sqrt
{\kappa^{2}+\tau^{2}}+\kappa}$ (\cite{bonnor}). In case of $\kappa=0,$ the
equation $\left(  \ref{n11}\right)  $ reduces to
\[
\alpha_{0}\left(  s\right)  =\frac{1}{\tau\sqrt{2}}\left(  \sinh\left(
\sqrt{\tau}s\right)  ,\cosh\left(  \sqrt{\tau}s\right)  ,\sin\left(
\sqrt{\tau}s\right)  ,-\cos\left(  \sqrt{\tau}s\right)  \right)  .
\]

\qquad The Cartan curvatures of the self-similar null curve $\gamma_{1}$ can
be given by $\kappa=0$ and $\tau=c\neq0$. Moreover, the Cartan curvatures of
$\gamma_{3}$ are $\kappa=ac$ and $\tau=c\neq0$ in the \textbf{Case} 3. Then,
the self-similar null curves $\gamma_{1}$ and $\gamma_{3}$ are null helices
which correspond to the null curves $\alpha_{0}$ and $\alpha$, respectively.
Hence, we can say that null helices satisfying $\tau\neq0$ are a class of
self-similar null curves in $\mathbb{M}^{4}$. Also, we can characterize\ the
null helices by means of the shape Cartan curvatures in $\mathbb{M}^{4}$.

\bigskip

Hakan Simsek and Mustafa \"{O}zdemir

Department of Mathematics

Akdeniz University

Antalya, TURKEY;

e-mail: hakansimsek@akdeniz.edu.tr.

\ \ \ \ \ \ \ \ \ \ \ mozdemir@akdeniz.edu.tr

\end{document}